\documentclass[11pt, a4paper, oneside]{amsart}
\usepackage{graphicx} 
\usepackage{color}
\usepackage{amsmath}
\usepackage{amscd}
\usepackage{amsthm,mathrsfs}
\usepackage{mathtools}
\usepackage{amssymb}
\usepackage[margin=1in]{geometry}
\theoremstyle{plain}
\newtheorem{theorem}{Theorem}[section]
\newtheorem{lemma}[theorem]{Lemma}
\newtheorem{corollary}[theorem]{Corollary}
\newtheorem{proposition}[theorem]{Proposition}

\newtheorem{remark}[theorem]{Remark}
\newtheorem{definition}[theorem]{Definition}

\title[Brownian motion and the non-compact Jacobi process]{Brownian motion on the complex general linear group and the non-compact Jacobi process}

\author[Martin Auer]{Martin Auer ${}^{1}$}
\address{${}^{1}$Fakult\"at Mathematik, Technische Universit\"at Dortmund, Vogelpothsweg 87, D-44221 Dortmund, Germany}
\email{martin.auer@tu-dortmund.de}

\author[Nizar Demni]{Nizar Demni ${}^{2}$}
\address{${}^{2}$Division of Science and Mathematics, New York University Abu Dhabi, P.O. Box 129188, Abu Dhabi, United Arab Emirates}
\email{nd2889@nyu.edu}

\author[Nicolas Gilliers]{Nicolas Gilliers ${}^{3}$}
\address{${}^{3}$ Université Paris Cité, CNRS, MAP5, F-75006 Paris, France}
\email{nicolas.gilliers@u-paris.fr}

\begin{document}
\begin{abstract}
In this paper, we introduce and study a new non Hermitian matrix-valued process built from truncations of Brownian motion on complex positive-definite matrices and of its inverse. We prove that its eigenvalue process is valued in $(1,+\infty)$ and that it forms a non-compact Jacobi particle system. 
Surprisingly, the same particle system occurs in \cite{BDW} in connection with Brownian motion on a non-compact complex Grassmann manifold. In this respect, we supply evidences showing that this occurrence holds true only at the spectral level and not on the corresponding matrix diffusions. 
We then study the large-size limit of our matrix model and prove results parallel to those obtained by the second author for the free Jacobi process.
\end{abstract}

\keywords{Brownian motion on the complex general linear group; Positive definite complex matrices; non-compact Jacobi particle system; free multiplicative Brownian motion; Free non-compact Jacobi process.}
\maketitle
\section{Introduction}

\subsection{Brownian motion on ellipsoids}
Let $(\Omega, \mathscr{F}, \mathbb{P})$ be a complete probability space and let $G$ be a left-invariant Brownian motion defined on $\Omega$ and valued in the complex general linear group $GL(N,\mathbb{C}), N \geq 1$ \cite{Liao}. This matrix-valued process is the unique strong solution of the stochastic differential equation (hereafter SDE):   
\begin{equation*}
dG(t) = G(t) \circ d\mathcal{Z}(t), \quad G_0 \in GL(N,\mathbb{C}), 
\end{equation*}
in Stratonovich sense, where $\mathcal{Z}$ is a $N \times N$ complex Brownian matrix whose entries have common variance $t/N$ (this normalization is needed for later purposes). Note that since the matrix-valued quadratic variation satisfies $(d\mathcal{Z})(d\mathcal{Z}) = 0$, then the same equation holds in It\^o's sense, namely $dG(t) = G(t) d\mathcal{Z}(t)$. We consider the squared radial part of $G$: 
\begin{equation*}
H:=GG^{\star}.    
\end{equation*} 
This is a Brownian motion on the space of complex positive definite matrices and satisfies the SDE: 
\begin{equation}\label{SDE0}
dH(t) = \sqrt{2}G(t)d\mathcal{N}(t)G(t)^{\star} + H(t) dt,    
\end{equation}
where $\mathcal{N}:= (\mathcal{Z} + \mathcal{Z}^{\star})/\sqrt{2}$ is a Hermitian (or Dyson) Brownian motion whose entries have common variance $t/N$. The real analogue of $H$ appeared in \cite{NRW} and was coined \emph{Brownian motion on ellipsoids}. The process $H$ also appeared in the mathematical literature. For instance, it was shown in \cite{Jones-Oco} that its eigenvalues (in increasing order) evolve like independent Brownian motions with drifts conditioned never to collide. For both the real and complex division algebras, the componentwise logarithm of the corresponding eigenvalue process is an instance of the radial Heckman--Opdam process of type $A$ \cite{Sch}, which was studied a few years earlier in \cite{Cep-Lep} without any reference to Lie algebra theory.

In the large-size limit as $N \rightarrow +\infty$, it was proved  in \cite{Cebron} that the averaged normalized traces of polynomials in $\{G, G^{\star}, G^{-1}, [G^{-1}]^{\star}\}$ converge to the non-commutative distribution of the free multiplicative Brownian motion $(g_t)_{t \geq 0}$ and its inverse, a time-dependent family of bounded operators in a non-commutative probability space. In particular, for any time $t > 0$, the normalized moments of $H_t$ converge as $N \rightarrow +\infty$ to those of the free positive Brownian motion
\[
h_t:= g_tg_t^{\star},\quad t \geq 0.
\]
The latter were computed in \cite{Biane} while the corresponding spectral distribution was described later in \cite{Biane1}. 

\subsection{The unitary Brownian motion and its truncations} 

The unitary group $\mathcal{U}(N)$ is the maximal compact subgroup of $GL(N,\mathbb{C})$ and carries a Brownian motion $(U_t)_{t \geq 0}$ whose distribution is invariant under the adjoint action. The eigenvalues of this process evolve like independent Brownian motions on the unit circle and conditioned never to collide \cite{Hob-Wer}. Moreover, a full expansion of products of averaged normalized traces was derived in \cite{Lev} and reveals a rich combinatorial structure stemming from the Schur-Weyl duality (see also \cite{Dah}). In the large-size limit as $N \rightarrow +\infty$, the unitary Brownian motion converges (in the sense of non-commutative moments) to the so-called free unitary Brownian motion. The latter was defined in \cite{Biane} as a unitary free L\'evy process and its spectral distribution was described in \cite{Biane1}.

Square truncations (in the present work, these are understood as upper-left corners) of the unitary Brownian motion are non-normal matrices. In the asymptotic regime as $t \rightarrow +\infty$, they converge to upper-left corners of 
unitary Haar-distributed matrices. Such matrix models have been extensively studied in various areas of mathematics and mathematical physics (see e.g. \cite{For}). As to their large-size limits (for fixed time $t$), they converge to compressions of the free unitary Brownian motion by a free orthogonal projection in a non-commutative probability space. 
In this regard, explicit formulas for various mixed moments of these non normal operators were obtained in \cite{Demni-Hamdi}, while the supports of their corresponding Brown measures were determined in \cite{Demni-Hamdi1}. 
Considering the squared radial part of a given truncation, one obtains the Hermitian Jacobi process whose moments at any given time were explicitly computed in \cite{Del-Dem} and \cite{DHS}. In the large $N$-limit, the asymptotics of these moments were further obtained in \cite{Dem-Ham} and describe the spectral distribution of the free Jacobi process \cite{Demni}.  

\subsection{Our contributions}
Let $1 \leq n \leq N$. In this paper, we consider $n \times n$ truncations $(\pi_n(H(t)))_{t \geq 0}$ of the complex Hermitian process $(H(t))_{t \geq 0}$. 
Since $\pi_n(H_t)$ is obviously complex Hermitian, we apply Bru's Theorem \cite{Kat-Tan} to derive the SDEs satisfied by its eigenvalue process. 
Using logarithmic coordinates, the transformed process comes with an additional constant drift $(N-n)t$ compared to the logarithmic coordinates of the eigenvalues of $H$. As a matter of fact, the eigenvalues of $\pi_n(H)$ never collide almost surely and enjoy the strong uniqueness property for any starting point. More generally, one may consider rectangular truncations of $H$ and their singular values, but we limit ourselves to square truncations for sake of simplicity. 

Motivated by the construction of the Hermitian Jacobi process from truncations of $(U_t)_{t \geq 0}$, we introduce and study a new matrix-valued process from truncations of $(H_t)_{t \geq 0}$ defined by\footnote{We omit the dependence on $N$ and $n$ for simplicity.}
\begin{equation*}
J(t):= \pi_n(H(t))\pi_n(H^{-1}(t)), \quad t \geq 0.    
\end{equation*}
This is not a complex Hermitian process, yet its eigenvalues are real since they coincide with those of the complex Hermitian process $[\pi_n(H^{-1})]^{1/2}\pi_n(H)[\pi_n(H^{-1})]^{1/2}$. 
By using convex matrix inequalities for the L\"owner order, we prove that the eigenvalues of $J(t)$ are greater than or equal to one (we also give a proof of this fact based on the Schur complement). 
The Schur-complement representation identifies the nontrivial spectrum and shows that, if $n>N/2$, at least $2n-N$ eigenvalues are identically equal to one. 
Afterwards, we extend Bru's theorem to complex diagonalizable (not necessarily Hermitian) matrix-valued processes and use it to derive the SDE satisfied by the eigenvalues of $J$. The latter form a non-compact Jacobi particle system \cite{Auer-Voit}, which coincides with the one arising in the study of generalized stochastic areas on complex hyperbolic Grassmann manifolds \cite{BDW}. In particular, we infer from \cite{BDW} that its components do not collide almost surely and that it is defined for all $t \geq 0$ from every admissible starting point. We then realize the non-compact Grassmann manifold as a generalized Poincar\'e disc. This yields an exact pointwise matrix-ball representation of the spectrum of $J$, while comparison of the induced generators shows that the two full matrix diffusions are different. Indeed, the Grassmannian metric involves the left and right singular-vector structures of a rectangular matrix, whereas the affine-invariant metric on the positive cone involves the eigenspaces of a positive-definite Hermitian matrix.

After this wave of stochastic calculus and differential geometry, we proceed to the study of the free probabilistic analogue of our matrix model $J$. In order to outline our results, recall that a non-commutative probability space is a pair $(\mathscr{A}, \tau)$ where $\mathscr{A}$ is a unital algebra endowed with a state $\tau$. In free probability theory, it is further assumed that $\mathscr{A}$ is a von Neumann algebra and that $\tau$ is tracial, faithful and normal. Back to our matrix model, we express the truncation operator $\pi_n$ as a compression by a diagonal $N \times N$ projection $P_n$ of rank $n$. Assuming that $n=n(N)$ satisfies
\[
\frac{n(N)}{N} \rightarrow \kappa \in (0,1], \quad N \rightarrow +\infty,
\]
and using Theorem 4.6 in \cite{Cebron}, we obtain the self-adjoint operator $Ph_tP$ as the large-size limit of the matrix compression. Here, $P$ is an orthogonal projection of trace $\tau(P) = \kappa$ and is free from $(h_t)_{t \geq 0}$ in $(\mathscr{A}, \tau)$. Using standard techniques from free probability theory \cite{Demni-Hamdi}, we prove that the spectral distribution of $Ph_tP$, viewed as an operator in the compressed probability space 
\begin{equation*}
(P\mathscr{A}P, \tau/\tau(P)),    
\end{equation*}
coincides with the spectral distribution of $e^{(1-\kappa)t}h_{\kappa t}$ in $(\mathscr{A}, \tau)$. Next, we define and study the operator-valued process 
\begin{equation*}
j_t:= Ph_tPh_t^{-1}P, \quad t\geq 0,    
\end{equation*}
which is the large-size limit of the matrix-valued process $J$. In particular, we study its spectrum in the compressed algebra $P\mathscr{A}P$ and prove that it lies in $[1,\infty)$.
By using free stochastic calculus, we derive an ordinary differential system satisfied by its moments. When $\kappa=1/2$, we further prove that the spectral distribution of $j_t$ is the pushforward of the spectral distribution of $h_t$ under a rational map. These results are analogues of those obtained in \cite{DHH} for the free Jacobi process.  

The paper is organized as follows. In section 2, we derive the SDE satisfied by the eigenvalue processes corresponding to square truncations of $H$. The matrix-valued process $J$ is defined in Section 3, where we derive the SDE satisfied by its eigenvalue process and also explain the connection between this particle system and the one already dealt with in \cite{BDW}. 
In section 4, we investigate the connection between the corresponding matrix-valued processes. Section 5 is concerned with large-$N$ limits of square truncations of $H$ and 
of the matrix-valued process $J$ in the vector space of complex $N\times N$ matrices endowed with its normalized trace. 

For clarity, we drop the dependence on $N$ in some notation. We hope this will not cause confusion.

{\bf Acknowledgment}: The matrix-valued $J$ and its large-size limit were defined in Auer's dissertation \cite{Auer-PhD}. There, the interested reader may find different proofs of some of our results.   

\section{Square Truncations of $H=GG^{\star}$}
Recall the $n \times n$  truncation $\pi_n(H)$ of $H$, see \eqref{SDE0}, then: 
 \begin{equation*}
 \pi_n(H) \oplus 0_{N-n} = P_nHP_n
 \end{equation*}
where $P_n$ is the diagonal projection of rank $n$. Since $\pi_n(H(t))$ is a complex Hermitian matrix, the so-called Bru's Theorem (\cite{Kat-Tan}, Theorem 1) is applied to write down the SDE satisfied by its (real) eigenvalues process. 
This Theorem applies to any complex Hermitian process and summarizes Bru's calculations from her study of real Wishart processes (\cite{Bru}). We only state it for $\pi_n(H(t))$. 
\begin{theorem}[\cite{Kat-Tan}] 
For any fixed $1 \leq n \leq N$ and $t > 0$, let $K_n(t)$ be the change of basis matrix defined by: 
\begin{equation*}
K_n(t)^{\star}\pi_n(H(t))K_n(t) = {\rm diag}(\lambda_1^{N,n}(t), \dots, \lambda_n^{N,n}(t)), \quad t \geq 0,
\end{equation*}
where the eigenvalues are ordered as: 
\begin{equation*}
\lambda_1^{N,n}(t)\geq \dots \geq \lambda_n^{N,n}(t).
\end{equation*}
Define further the quadratic variation: 
\begin{equation*}
\Gamma_{ij}^{N,n}(t)dt := (K_n(t)^{\star}d\pi_n(H(t))K_n(t))_{ij}(K_n(t)^{\star}d\pi_n(H(t))K_n(t))_{ji}. 
\end{equation*}
Then the eigenvalues process of $\pi_n(H)$ satisfies the SDE:
\begin{equation*}
d\lambda_i^{N,n}(t) = dM_i^{N,n}(t) + \sum_{j=1, j \neq i}^n \frac{\Gamma_{ij}^{N,n}(t)}{\lambda_i^{N,n}(t) - \lambda_j^{N,n}(t)} dt + dS_i^{N,n}(t), \quad 1 \leq i \leq n, 
\end{equation*}
up to the first collision time: 
\begin{equation*}
 \inf\{t > 0, \lambda_{i}^{N,n}(t) = \lambda_j^{N,n}(t) \,\, \textrm{for some} \, (i,j)\}.
\end{equation*}
 Here, $M_i^{N,n}$ is the local martingale whose quadratic variation is given by 
\begin{equation*}
\Gamma_{ii}^{N,n}(t) dt     
\end{equation*}
and $dS_i^{N,n}(t)$ is the finite-variation part of $(K_n(t)^{\star}d\pi_n(H(t))K_n(t))_{ii}$. Furthermore, for any $i \neq j$,
\begin{equation*}
(d\lambda_i^{N,n}(t))(d\lambda_j^{N,n}(t)) = (K_n(t)^{\star}d\pi_n(H(t))K_n(t))_{ii}(K_n(t)^{\star}d\pi_n(H(t))K_n(t))_{jj}.
\end{equation*}
\end{theorem}
Applying this Theorem, we readily get:
\begin{proposition}\label{Eig-Trun}
For any $1 \leq i \leq n$,  
\begin{equation*}
d\lambda_i^{N,n}(t) = \frac{\sqrt{2}}{\sqrt{N}}\lambda_i^{N,n}(t) dB_i(t) + \lambda_i^{N,n}(t) dt + 
\frac{2}{N}\sum_{j=1, j \neq i}^n \frac{\lambda_i^{N,n}(t)\lambda_j^{N,n}(t)}{\lambda_i^{N,n}(t) - \lambda_j^{N,n}(t)} dt
\end{equation*}
up to the first collision time. Here $(B_1, B_2, \dots, B_n)$ is a $n$-dimensional Euclidean Brownian motion.     
\end{proposition}
\begin{proof}
We shall derive this SDE up to the first collision time assuming the initial condition has distinct components. Afterwards, we shall prove that this random time is almost surely infinite and that there is no need for the initial condition to have distinct components. We start with 
\begin{equation*}
d\pi_n(H(t)) = \sqrt{2}P_nG(t) d\mathcal{N}(t)G(t)^{\star}P_n + \pi_n(H(t)) dt,   
\end{equation*}
which is an immediate consequence of \eqref{SDE0}. It follows that
\begin{align*}
K_n(t)^{\star}d\pi_n(H(t))K_n(t) & = \sqrt{2}K_n(t)^{\star}P_nG(t) d\mathcal{N}(t)G(t)^{\star}P_nK_n(t) 
+ \textrm{diag}(\lambda_1^{N,n}(t), \dots, \lambda_n^{N,n}(t)) dt. 
\end{align*}
Set $L_n(t) := G(t)^{\star}P_nK_n(t)$, then  
\begin{equation*}
\Gamma_{ij}^{N,n}(t)dt = 2(L_n(t)^{\star}d\mathcal{N}(t)L_n(t))_{ij}(L_n(t)^{\star}d\mathcal{N}(t)L_n(t))_{ji}. 
\end{equation*}
Recall that the quadratic variation of any two entries of the Dyson Brownian motion $\mathcal{N}$ is:
\begin{equation}\label{QVHBM}
 (d\mathcal{N}(t))_{rs}(d\mathcal{N}(t))_{mq} = \frac{1}{N}\delta_{rq}\delta_{ms}dt, 
\end{equation}
we infer
\begin{equation*}
\Gamma_{ii}^{N,n}(t)dt = 2[(L_n(t)^{\star}L_n(t))_{ii}]^2 dt = \frac{2[\lambda_i^{N,n}(t)]^2}{N} dt. 
\end{equation*}
Consequently, for any $1 \leq i \leq n$, there exists a real Brownian motion $B_i$ such that $dM_i^{N,n}(t) = \sqrt{2/N}\,\lambda_i^{N,n}(t) dB_i(t)$. 
Besides, for any $i \neq j$,
\begin{align*}
(d\lambda_i^{N,n}(t))(d\lambda_j^{N,n}(t)) &= (L_n(t)^{\star}d\mathcal{N}(t)L_n(t))_{ii}(L_n(t)^{\star}d\mathcal{N}(t)L_n(t))_{jj}
\\& = \frac{1}{N}(L_n(t)^{\star}L_n(t))_{ij}(L_n(t)^{\star}L_n(t))_{ji} = 0.
\end{align*}
As a matter of fact, the Brownian motions $(B_1,\dots, B_n)$ are mutually independent. Likewise,
\begin{equation*}
\Gamma_{ij}^{N,n}(t)dt = (L_n(t)^{\star}L_n(t))_{ii} (L_n(t)^{\star}L_n(t))_{jj} = \frac{2\lambda_i^{N,n}(t) \lambda_j^{N,n}(t)}{N} dt
\end{equation*}
for any $j \neq i$. Since $dS_i^{N,n}(t) = \lambda_i^{N,n}(t) dt$, the SDE is derived up to the first collision time. 

Next, we perform a logarithmic change of variables and define:  
\begin{equation*}
\gamma_i^{N,n}(t) := \frac{1}{2}\ln(\lambda_i^{N,n}(t)).   
\end{equation*}
This definition makes sense since one easily proves that $\det(\pi_n(H_t))$ is a geometric Brownian motion with positive drift (up to the time change $t \mapsto 2tn/N$). Therefore, It\^o's formula yields: 
\begin{align*}
d\gamma_i^{N,n}(t) & = \frac{1}{\sqrt{2N}}dB_i(t) + \frac{dt}{2} +
\frac{1}{N}\sum_{j=1, j \neq i}^n \frac{\lambda_j^{N,n}(t)}{\lambda_i^{N,n}(t) - \lambda_j^{N,n}(t)} dt  - \frac{dt}{2N} 
\\& = \frac{1}{\sqrt{2N}} dB_i(t) + \frac{1}{2N}\sum_{j=1, j \neq i}^n \coth\left(\gamma_i^{N,n}(t)-\gamma_j^{N,n}(t)\right)dt + \frac{N-n}{2N}dt.
\end{align*}
In particular, if $N=n$ then 
\begin{align*}
d\gamma_i^{N,N}(t) = \frac{1}{\sqrt{2N}} dB_i(t) 
+ \frac{1}{2N}\sum_{j=1, j \neq i}^N \coth\left(\gamma_i^{N,N}(t)-\gamma_j^{N,N}(t)\right)dt,
\end{align*}
which is the SDE satisfied by the eigenvalues process of $H = GG^{\star}$. As a result, the shifted process 
\begin{equation*}
 \gamma_i^{N,n}(2Nt) - (N-n)t, \quad t \geq 0, \quad 1 \leq i \leq n,   
\end{equation*}
satisfies the same SDE as $(\gamma_i^{n,n}(2nt))_{t \geq 0}, 1 \leq i \leq n$. According to Theorem 7.1 in \cite{Cep-Lep}, this SDE has a unique strong solution for all $t \geq 0$ and for any initial condition, and there is no collision almost surely. The proposition is proved. 
\end{proof}

\begin{remark}
Consider the rescaled process $\tilde{\gamma}_i^{N,n}(t) := \gamma_i^{N,n}(2Nt), 1 \leq i\leq n$. Then 
\begin{align*}
d\tilde{\gamma}_i^{N,n}(t) & =  dB_i(t) + \sum_{j=1, j \neq i}^n \coth\left(\tilde{\gamma}_i^{N,n}(t)-\tilde{\gamma}_j^{N,n}(t)\right)dt + (N-n)dt.
\end{align*}
Then, it is readily seen that the drift increases by a unit as the size of the truncation decreases by a unit. This simple fact is in agreement with Proposition 3.20 in \cite{AOW} due to the interlacing property satisfied by the eigenvalues. 
\end{remark}
\section{Our new matrix model and its eigenvalues process}
In this section, we introduce the matrix-valued process $J$ and give two proofs that its eigenvalues lie in $[1,\infty)$ at every time $t \geq 0$. The first proof relies on matrix inequalities, while the second uses a normalized Schur complement. The latter also identifies the rank defect when $n>N/2$.
Next, we extend Bru's theorem to complex diagonalizable matrices with real eigenvalues and use it to prove that the nontrivial eigenvalues of $J$ evolve as a non-compact Jacobi particle system. We then connect this particle system to the one studied in \cite{BDW}; in the regime without forced unit eigenvalues, this implies that the first collision time is almost surely infinite and that strong uniqueness holds for every admissible starting point.
\begin{definition}
For any $1 \leq n \leq N$, we define 
\begin{equation*}
 J(t) := \pi_n(H(t))\pi_n(H^{-1}(t)), \quad t \geq 0.   
\end{equation*}
\end{definition}
The following properties are straightforward to prove: 
\begin{itemize}
 \item The inverse $H_t^{-1}$ of $H_t$ exists since $G$ is invertible.
 \item Since $H_t^{-1}$ is almost surely positive definite then so is $\pi_n(H_t^{-1})$ and it follows that the eigenvalues of $J_t$ are the same as those of 
 \begin{equation*}
 [\pi_n(H_t^{-1})]^{1/2}\pi_n(H_t)[\pi_n(H_t^{-1})]^{1/2}.     
 \end{equation*}
 In particular, they are almost surely positive. 
 \end{itemize}

\begin{lemma}\label{Positivity}
For any time $t \geq 0$, the eigenvalues of $J$ are greater than or equal to one.
\end{lemma}
\begin{proof}
The statement of the lemma is indeed true for any positive definite matrix $A$ and we supply two proofs. 

\begin{itemize}
\item The first one amounts to prove the inequality
 \begin{equation}\label{Ineq}
 [\pi_n(A^{-1})]^{1/2}\pi_n(A)[\pi_n(A^{-1})]^{1/2} \geq I_n, 
 \end{equation}
 where $\geq$ is the L\"owner order. To proceed, we combine the results of exercices V.1.15 and V.2.2 in \cite{Bhatia} to deduce that 
 \begin{equation*}
\pi_n(A) \geq [\pi_n(A^{-1})]^{-1}. 
 \end{equation*} 
The desired inequality \eqref{Ineq} now follows then from Lemma V. 1.5 in the same book. 

\item The second proof uses the Schur complement and also identifies the multiplicity of the eigenvalue one. Write
\begin{equation*}
A = \begin{pmatrix}  A_{11} & A_{12} \\ A_{12}^{\star} & A_{22} 
\end{pmatrix}     
\end{equation*}    
where $A_{11} = A_{11}^{\star} \in \mathbb{C}^{n\times n}, A_{22} = A_{22}^{\star} \in \mathbb{C}^{(N-n)\times (N-n)},$
and $A_{12}\in\mathbb{C}^{n\times(N-n)}$. Define
\begin{equation}\label{Def-R}
R:=A_{11}^{-1/2}A_{12}A_{22}^{-1/2},
\end{equation}
then congruence of $A$ by $\operatorname{diag}(A_{11}^{-1/2},A_{22}^{-1/2})$ shows that the matrix
\begin{equation*}
\begin{pmatrix}I_n&R\\R^{\star}&I_{N-n}\end{pmatrix} 
\end{equation*}
is positive definite. Consequently, the Schur complement of its lower-right identity block 
\begin{equation}\label{R-contraction}
I_n-RR^{\star}
\end{equation}
is also positive definite, so every singular value of $R$ is strictly smaller than one. Moreover, the formula for the inverse of a block matrix entails
\begin{equation}\label{Schur-J}
A_{11}^{1/2}\pi_n(A^{-1})A_{11}^{1/2} = A_{11}^{1/2}(A_{11}-A_{12}A_{22}^{-1} A_{12}^{\star})^{-1}A_{11}^{1/2} =(I_n-RR^{\star})^{-1}.
\end{equation}
But the left-hand side of \eqref{Schur-J} is similar to $\pi_n(A)\pi_n(A^{-1}) = A_{11}\pi_n(A^{-1})$, therefore all its eigenvalues are at least one, as claimed.
\end{itemize}
\end{proof}

\begin{remark}\label{Rank-defect}
Let $q=N-n$ and $r=\min(n,q)$. For every positive-definite $A$, the matrix $\pi_n(A)\pi_n(A^{-1})$ has at least $n-r=(2n-N)_{+}$ eigenvalues equal to one. Its remaining eigenvalues are
\begin{equation*}
\frac{1}{1-s_1(R)^2},\ldots,\frac{1}{1-s_r(R)^2},
\end{equation*}
where $s_1(R),\ldots,s_r(R)$ are the possibly nonzero singular values of the matrix $R$ in \eqref{Def-R}. Moreover, the eigenvalues different from one coincide with those of the complementary product $A_{22}\pi_q(A^{-1})$.
Indeed, equation \eqref{Schur-J} reduces the assertion to the spectrum of $RR^{\star}$. The claim follows since $\operatorname{rank}(R)\leq r$ (so its kernel has dimension at least $n-r$) and since the nonzero eigenvalues of $RR^{\star}$ and $R^{\star}R$ agree with multiplicities. 
In particular,, when $R$ has full  rank the deterministic multiplicity of the eigenvalue one is exactly $(2n-N)_{+}$. Consequently, we shall assume in the sequel without loss of generality that $n \geq N/2$ so that all the eigenvalues of $J_t$ are almost surely strictly larger than one. 
\end{remark}

In order to write down the SDE satisfied by the eigenvalue process of $J$, suppose that the matrix $J(0)$ has distinct eigenvalues and let $T$ be the first collision time of any two eigenvalues. 
Then, for any $t < T,$ the matrix $J(t)$ is diagonalizable with real eigenvalues. The extension of Bru's Theorem for such matrices differs from it in that the `stochastic logarithm' of the matrix of change of basis is no longer skew-Hermitian (or skew symmetric in the real case). Nonetheless, we shall see in the proof below that the derivation of the SDE satisfied by the eigenvalue process of any complex diagonalizable matrix with real eigenvalues does not involve these diagonal terms. As a matter of fact, Bru's Theorem still holds true, replacing the conjugate transpose of the change of basis matrix with its inverse. 

\begin{theorem}\label{Bru-IM}
Let $\mathcal{X}(t), t \geq 0,$ be a diagonalizable complex $N \times N$ matrix with real eigenvalues and denote $Q(t), t\geq 0,$ the change of basis matrix: 
\begin{equation*}
Q^{-1}(t)\mathcal{X}(t)Q(t) = \textrm{diag}(\lambda_1(t), \dots, \lambda_N(t)) := \Lambda(t).    
\end{equation*}
Define the quadratic variation 
\begin{equation*}
\Theta_{ij}(t)dt := (Q^{-1}(t)d\mathcal{X}(t)Q(t))_{ij}(Q(t)^{-1}d\mathcal{X}(t)Q(t))_{ji}. 
\end{equation*}
Then the eigenvalues process of $\mathcal{X}$ satisfies, up to the first collision time, the SDE: 
\begin{equation*}
d\lambda_i(t) = dF_i(t) + \sum_{j=1, j \neq i}^N \frac{\Theta_{ij}(t)}{\lambda_i(t) - \lambda_j(t)} dt + dZ_i(t),   
\end{equation*}
where $dF_i$ is the local martingale whose quadratic variation is given by 
\begin{equation*}
\Theta_{ii}(t) dt     
\end{equation*}
and $dZ_i$ is the finite-variation part of $(Q^{-1}d\mathcal{X}Q)_{ii}$. Moreover, 
\begin{equation*}
(d\lambda_i)(d\lambda_j) = (Q^{-1}d\mathcal{X}Q)_{ii}(Q^{-1}d\mathcal{X}Q)_{jj}.    
\end{equation*}
\end{theorem} 
\begin{proof} 
Since the matrix of change of basis is not necessarily unitary, we mimic the proof of Bru's Theorem replacing the conjugate transpose by the inverse. More precisely, introduce the following matrix-valued processes (`stochastic logarithms'):  
\begin{eqnarray*}
dA & := & Q^{-1}dQ + \frac{1}{2}(dQ^{-1})(dQ), \\ 
dB & := & dQ^{-1}Q + \frac{1}{2}(dQ^{-1})(dQ).
\end{eqnarray*}
Then $Q^{-1}Q = I_N$ implies that $dA+dB= 0$. Moreover, we readily see that 
\begin{equation*}
(dB)(dA) = (dQ^{-1})(dQ),    
\end{equation*}
whence 
\begin{equation}\label{Rel1}
dQ = Q\left[dA - \frac{1}{2}(dB)(dA)\right] = Q\left[dA + \frac{1}{2}(dA)(dA)\right].   
\end{equation}
Similarly 
\begin{equation}\label{Rel2}
dQ^{-1} = \left[dB - \frac{1}{2}(dB)(dA)\right]Q^{-1} = 
\left[-dA + \frac{1}{2}(dA)(dA)\right]Q^{-1}.   
\end{equation}
Now, use It\^o's formula to write 
\begin{align*}
d\Lambda & = [dQ^{-1}]\mathcal{X}Q + Q^{-1}[d\mathcal{X}]Q + Q^{-1}\mathcal{X}[dQ] 
\\& + (dQ^{-1})\mathcal{X}(dQ) + Q^{-1}(d\mathcal{X})(dQ) + (dQ^{-1})(d\mathcal{X})Q.
\end{align*}
Using \eqref{Rel1} and \eqref{Rel2}, we get the analogue of equation (3.5) in \cite{Bru}, namely:
\begin{align}\label{Eigenvalues}
d\Lambda &= Q^{-1}d\mathcal{X}Q + [\Lambda dA - dA \Lambda]- (dA)\Lambda (dA) 
\\& + Q^{-1}(d\mathcal{X})Q(dA) - (dA)Q^{-1}(d\mathcal{X})Q 
+\frac{1}{2}[(dA)(dA)\Lambda + \Lambda(dA)(dA)] \nonumber.
\end{align}
Equating diagonal entries of both sides, we get the quadratic variation
\begin{equation}\label{Ident1}
(d\lambda_i)(d\lambda_i) = (Q^{-1}d\mathcal{X}Q)_{ii}(Q^{-1}d\mathcal{X}Q)_{ii},   
\end{equation}
since the diagonal part of $\Lambda [dA] - [dA] \Lambda$ vanishes. More generally, for any $1 \leq i,j \leq N$,
\begin{equation}\label{Ident2}
(d\lambda_i)(d\lambda_j) = (Q^{-1}d\mathcal{X}Q)_{ii}(Q^{-1}d\mathcal{X}Q)_{jj}.   
\end{equation}
By considering non-diagonal terms of both sides, we get the following identity:
\begin{equation*}
 [\lambda_j-\lambda_i]dA_{ij} = [Q^{-1}d\mathcal{X}Q]_{ij} +  \textrm{Finite variation terms}.
\end{equation*}
Consequently, the quadratic variation 
\begin{equation*}
[\lambda_j-\lambda_i][\lambda_m-\lambda_k] (dA_{ij})(dA_{km})   
\end{equation*}
is given by 
\begin{equation*}
(Q_t^{-1}d\mathcal{X}_tQ_t)_{ij}(Q_t^{-1}d\mathcal{X}_tQ_t)_{km}.
\end{equation*}
As a result,
\begin{align}\label{Ident3}
\frac{1}{2}[(dA)(dA)\Lambda + \Lambda(dA)(dA)]_{ii} - [(dA)\Lambda (dA)]_{ii} 
& =  \sum_{r \neq i}(\lambda_i-\lambda_r)(dA)_{ir}(dA)_{ri}  \nonumber \\ &
 = \sum_{r \neq i}\frac{ (Q^{-1}d\mathcal{X}Q)_{ir}(Q^{-1}d\mathcal{X}Q)_{ri}}{\lambda_r-\lambda_i},
 \end{align}
and similarly
\begin{align}\label{Ident4}
[Q^{-1}(d\mathcal{X})Q(dA) - (dA)Q^{-1}(d\mathcal{X})Q]_{ii} = 2\sum_{r \neq i} 
\frac{(Q^{-1}d\mathcal{X}Q)_{ir}(Q^{-1}d\mathcal{X}Q)_{ri}}{\lambda_i-\lambda_r}.
\end{align}
In particular, there is no contribution of the diagonal terms $dA_{ii}$. Remembering \eqref{Eigenvalues}, \eqref{Ident1}, \eqref{Ident2}, \eqref{Ident3}, \eqref{Ident4}, the Theorem is proved. 
\end{proof}

\subsection{Autonomous SDEs for $H$ and $H^{-1}$} 
Before proceeding to the derivation of the SDE satisfied by the eigenvalue process of $J$, we find it informative to derive autonomous SDEs for both processes $H$ and $H^{-1}$. 
To proceed, we define for any time $t \geq 0$:
\begin{equation*}
V(t):= H^{-1/2}(t)G(t).     
\end{equation*}
Then 
\begin{equation*}
    V(t)V(t)^{\star} = H^{-1/2}(t)G(t)G(t)^{\star}H^{-1/2}(t) = I_N
    \end{equation*}
so that $V(t)$ is a unitary matrix. Now, recall the SDE \eqref{SDE0}:
\begin{equation*}
dH(t) = \sqrt{2}G(t)d\mathcal{N}(t)G(t)^{\star} + H(t) dt,    
\end{equation*}
where $(\mathcal{N}(t))_{t \geq 0}$ is a Hermitian Brownian matrix whose entries have common variance $t/N$. Using L\'evy's characterization of multidimensional real Brownian motions and the quadratic variation \eqref{QVHBM}, lengthy (but easy) computations show that the process $(\mathcal{R}(t))_{t \geq 0}$ defined by 
\begin{equation*}
\mathcal{R}(t) := \int_0^tV(s)d\mathcal{N}(s)V(s)^{\star}
\end{equation*}
is also a Hermitian Brownian motion of variance $t/N$. 
\begin{proposition}
The matrix-valued processes $H$ and $H^{-1}$ satisfy the following SDEs:     
\begin{equation*}
dH(t) = \sqrt{2}H^{1/2}(t)d\mathcal{R}(t)H^{1/2}(t) + H(t) dt,
\end{equation*}
\begin{equation*}
dH^{-1}(t) = -\sqrt{2}H^{-1/2}(t)d\mathcal{R}(t) H^{-1/2}(t) + H^{-1}(t) dt.     
\end{equation*}
\end{proposition}
Note that both SDEs are almost the same (except that their driving Brownian motions have opposite trajectories). This is in agreement with the fact that $H^{-1} = [G^{\star}]^{-1}G^{-1}$ and that $(G^{\star})^{-1}$ is also a left-invariant Brownian motion in $GL(N,\mathbb{C})$.

From these SDEs and the definition $J(t) = \pi_n(H(t))\pi_n(H^{-1}(t))$, we may write 
\begin{equation*}
J(t) \oplus 0_{N-n} = P_nH(t)P_nH^{-1}(t)P_n    
\end{equation*}
and apply It\^o's formula to get: 
\begin{align}\label{Mat-SDE}
dJ(t) \oplus 0_{N-n} & = [\sqrt{2}P_nH^{1/2}(t)]d\mathcal{R}(t)[H^{1/2}(t)P_nH^{-1}(t)P_n] 
\nonumber \\& - [P_nH(t)P_nH(t)^{-1/2}]d\mathcal{R}(t)[\sqrt{2}H^{-1/2}(t)P_n]  
+ 2\left[J_t - \frac{n}{N}I_n\right] \oplus 0_{N-n}\, dt. 
\end{align}
We are now ready to apply Theorem \ref{Bru-IM} to the $n\times n$ process $J$.  

\subsection{Eigenvalue process of $J$}
Recall the assumption $n\leq N/2$ and suppose that $J(0)$ has distinct eigenvalues. Denote $T$ be the first collision time and for $t<T$, choose an invertible matrix $\widetilde Q(t)$ such that
\begin{equation*}
\widetilde Q^{-1}(t)J(t)\widetilde Q(t)=\operatorname{diag}(\lambda_1(t),\ldots,\lambda_n(t)).
\end{equation*}
For the covariance computation only, it is convenient to introduce the block matrix
\begin{equation*}
Q(t):=\widetilde Q(t)\oplus I_{N-n},\qquad
\Lambda(t):=\operatorname{diag}(\lambda_1(t),\ldots,\lambda_n(t),0,\ldots,0).
\end{equation*}
The repeated zero eigenvalue in this auxiliary embedding is not used in Theorem \ref{Bru-IM}; that theorem is applied directly to $J$ and $\widetilde Q$.

\begin{theorem}\label{J-eigenvalue-SDE}
For any $1 \leq i \leq n$ and any $t < T$,
 \begin{align}
d\lambda_i(t) & = \frac{2\sqrt{\lambda_i(t)(\lambda_i(t)-1)}}{\sqrt{N}}dB_i(t) \nonumber\\
&\quad + \frac{2}{N}\left[N\lambda_i(t) - n 
+ \sum_{\substack{1\leq j\leq n\\j \neq i}}\frac{2\lambda_i(t)\lambda_j(t) - \lambda_i(t) - \lambda_j(t)}{\lambda_i(t)-\lambda_j(t)} \right] dt \nonumber\\
&= \frac{2\sqrt{\lambda_i(t)(\lambda_i(t)-1)}}{\sqrt{N}}dB_i(t)\nonumber\\
&\quad+\frac{2}{N}\left[(N-2n+2)\lambda_i(t)-1
+2\lambda_i(t)(\lambda_i(t)-1)\sum_{\substack{1\leq j\leq n\\j\neq i}}
\frac{1}{\lambda_i(t)-\lambda_j(t)}\right]dt,
\label{J-SDE-n}
\end{align}
where $(B_1, \dots, B_n)$ are mutually independent real Brownian motions. 
\end{theorem}
\begin{proof}
From \eqref{Mat-SDE}, the embedded matrix $\mathcal{X}:=J\oplus0_{N-n}$ satisfies
\begin{align*}
d\mathcal{X}(t) & = \alpha(t) d\mathcal{R}(t)\beta(t) - \theta(t) d\mathcal{R}(t) \delta(t) 
+ 2\left[J(t) - \frac{n}{N}I_n\right] \oplus 0_{N-n}\,dt.
\end{align*}    
where we set:
\begin{eqnarray*}
\alpha &=& \sqrt{2}P_nH^{1/2}, \\     
\beta &=& H^{1/2}P_nH^{-1}P_n, \\ 
\theta &=& P_nHP_nH^{-1/2}, \\
\delta &=& \sqrt{2}H^{-1/2}P_n.
\end{eqnarray*}
It follows that: 
\begin{align*}
Q^{-1}(t) d\mathcal{X}(t)Q(t) & = Q^{-1}(t)\alpha(t) d\mathcal{R}(t)\beta(t)Q(t) - Q^{-1}(t)\theta(t) d\mathcal{R}(t) \delta(t) Q(t) 
+ 2\left[\Lambda(t) - \frac{n}{N}P_n\right] dt.
\end{align*}    
Now recall that $(R(t))_{t \geq 0}$ is a Hermitian Brownian motion whose entries have common variance $t/N$ and that the quadratic variation of the local martingale part of $d\lambda_i(t)$ 
coincides with the one of 
\begin{align*}
(Q(t)^{-1} d\mathcal{X}(t)Q(t))_{ii} = (Q^{-1}(t)\alpha(t) d\mathcal{R}(t)\beta(t)Q(t) - Q^{-1}(t)\theta(t) d\mathcal{R}(t) \delta(t) Q(t))_{ii}, 
\quad 1 \leq i \leq n.
\end{align*}
Using $(d\mathcal{R})_{l_1l_2}(d\mathcal{R})_{l_3l_4} = \delta_{l_1l_4}\delta_{l_2l_3}\,dt/N$ and multiplying the four possible products of the two martingale terms gives
\begin{equation*}
N(d\mathcal{X}(t))_{rs}(d\mathcal{X}(t))_{lq} = \left\{4 (\mathcal{X}(t))_{rq} (\mathcal{X}(t))_{ls} -2 (P_n)_{ls}(\mathcal{X}(t))_{rq} -2 (P_n)_{rq}(\mathcal{X}(t))_{ls}\right\} dt.    
\end{equation*}
As a result, one obtains for any $1 \leq i,j,k,m \leq N$ the identity:
\begin{align*}
&N (Q^{-1}(t)d\mathcal{X}(t)Q(t))_{ij}
   (Q^{-1}(t)d\mathcal{X}(t)Q(t))_{km}\\
&\qquad= \left\{4\Lambda_{im}(t)\Lambda_{kj}(t)
- 2(P_n)_{kj}\Lambda_{im}(t) - 2 (P_n)_{im}\Lambda_{kj}(t)\right\} dt. 
\end{align*}
Specializing this identity to $1 \leq i=j=k=m \leq n$, we get  
\begin{equation*}
(d\lambda_i(t))(d\lambda_i(t)) = \frac{4[\lambda_i(t)]^2 - 4\lambda_i(t)}{N}dt = \frac{4\lambda_i(t)(\lambda_i(t)-1)}{N} dt, 
\end{equation*}
while specializing it to $1 \leq i = m, k = j \leq n, i \neq j,$ and applying Theorem \ref{Bru-IM} directly to $J$ shows that the finite-variation part of $d\lambda_i(t)$ is
\begin{align*}
2\left[\lambda_i(t) - \frac{n}{N}\right]dt + \frac{1}{N}\sum_{\substack{1\leq j\leq n\\j \neq i}}\frac{4\lambda_i(t)\lambda_j(t) - 2\lambda_i(t) - 2\lambda_j(t)}{\lambda_i(t)-\lambda_j(t)} dt.
\end{align*}
Finally, specializing once again the same identity to $1 \leq i = j, k = m \leq n, i \neq k,$ implies that 
\begin{equation*}
(d\lambda_i(t))(d\lambda_k(t)) = 0,
\end{equation*}
so that the corresponding driving Brownian motions are mutually independent. Finally, the identity
\begin{equation*}
\frac{2\lambda_i\lambda_j-\lambda_i-\lambda_j}{\lambda_i-\lambda_j}
=1-2\lambda_i+\frac{2\lambda_i(\lambda_i-1)}{\lambda_i-\lambda_j}
\end{equation*}
gives the second form \eqref{J-SDE-n}.
\end{proof}  


\subsection{Relation to singular values of Brownian motion on the complex hyperbolic Grassmann manifold}
The eigenvalue process $(\lambda_i,1\leq i\leq n)$ is referred to as the non-compact Jacobi process \cite{Auer-Voit}. It is the non-compact twin of the eigenvalue process of the Hermitian Jacobi process \cite{Del-Dem}. Up to an affine transformation, it also arises from a corner of Brownian motion on the indefinite unitary group $U(N-n,n), 2n \leq N$, or equivalently from Brownian motion on the associated non-compact symmetric space \cite{BDW}. Indeed, setting
\begin{equation*}
\rho_i(t) = 2\lambda_i(Nt)-1, \quad 1 \leq i \leq n, 
\end{equation*}
and using \eqref{J-SDE-n}, we obtain
\begin{equation}\label{NCJ1}
 d\rho_i(t) = 2\sqrt{\rho_i^2(t)-1}dB_i(t) + 2[(a+2)\rho_i(t) + a]dt + 4(\rho_i^2(t)-1)\sum_{\substack{1\leq j\leq r\\j\neq i}}\frac{dt}{\rho_i(t)-\rho_j(t)}.   
\end{equation}
This is the radial particle system associated with Brownian motion on the complex hyperbolic Grassmann manifold \cite[Proposition 2.7]{BDW}
\begin{equation*}
\textrm{HG}^{n,N-n}:= U(N-n,n)/(U(N-n) \times U(n)).    
\end{equation*} 
In particular, the first collision time is almost surely infinite, the nontrivial particles are strictly larger than one for $t>0$, and the system has a unique strong solution from every point of the closed Weyl chamber \cite[Corollary 2.8]{BDW}. Moreover, $(\operatorname{arcosh}(\rho_i(t/4)))_{1\leq i\leq n}$ is a radial Heckman--Opdam process associated with the root system $BC_n$; see \cite[equation (17)]{BDW}.

\section{The matrix picture}
In this section we assume $n\leq N/2$; Proposition \ref{Rank-defect} reduces the complementary regime to this one. Since we already connected the eigenvalue process $(\lambda_i,1\leq i\leq n)$ to $(\rho_i,1\leq i\leq n)$ constructed in \cite{BDW}, we now examine the connection at matrix level. 
To this end, let us recall from \cite{BDW} how the process $(\rho_i,1\leq i\leq n)$ arose from the Brownian motion on the indefinite Lie group. The latter is defined by: 
\begin{equation*}
\mathbf{U}(N-n,n)=\left\{ M \in GL(N,\mathbb{C}), M^* \begin{pmatrix} I_{N-n}   & 0 \\ 0 & -I_{n} \end{pmatrix} M =  \begin{pmatrix} I_{N-n}   & 0 \\ 0 & -I_{n} \end{pmatrix}\right\},
\end{equation*}
and its Lie algebra is given by: 
\begin{equation*}
    \mathfrak{u}(N-n,n)=\left\{ A \in \mathbb{C}^{N \times N}, A^* \begin{pmatrix} I_{N-n}   & 0 \\ 0 & -I_{n} \end{pmatrix} + \begin{pmatrix} I_{N-n}   & 0 \\ 0 & -I_{n} \end{pmatrix}A= 0 \right\}.
\end{equation*}
Let $U$ be the matrix-valued process satisfying the Stratonovich differential equation\footnote{This is the Brownian motion on $\mathbf{U}(N-n,n)$.}
\begin{equation}\label{eq-bm-sde}
\begin{cases}
dU(t)=U(t)\circ d\mathcal{S}(t) ,\\
U_0=\begin{pmatrix}  Y_0 & X_0 \\ W_0 & Z_0 \end{pmatrix},
\end{cases}
\end{equation}
where $\mathcal{S}$ is Brownian motion in the Lie algebra $\mathfrak{u}(N-n,n)$. 
Write the block decomposition
\begin{equation*}
U(t) = \begin{pmatrix}  Y(t) & X(t) \\ W(t) & Z(t) \end{pmatrix}    
\end{equation*}
where $X(t) \in \mathbb{C}^{(N-n)\times n}, Y(t) \in \mathbb{C}^{(N-n)\times (N-n)}, 
W(t) \in \mathbb{C}^{n \times (N-n)}, Z(t) \in \mathbb{C}^{n\times n}$. Then $Z(t)$ is  invertible since 
\begin{equation*}
Z^{\star}Z = I_n + X^{\star}X.   
\end{equation*}
Therefore, the following two processes are well defined: 
\begin{equation*}
w := XZ^{-1}, \quad \mathcal{J} := w^{\star}w.
\end{equation*}
The former is Brownian motion on $\textrm{HG}^{n,N-n}$ in matrix-ball coordinates, constructed from the matrix analogue of projective coordinates in anti-de Sitter space \cite{Wang}; the latter is its squared radial part. 
In \cite{BDW}, we proved that the process $\mathcal{J}$ satisfies 
\begin{multline}\label{SDE-CHG}
d\mathcal{J}(t) = \sqrt{I_n-\mathcal{J}(t)}d\tilde{\mathcal{Z}}(t)\sqrt{I_n-\mathcal{J}(t)}\sqrt{\mathcal{J}(t)} + \sqrt{\mathcal{J}(t)}\sqrt{I_n-\mathcal{J}(t)}d\tilde{\mathcal{Z}}(t)^{\star} 
\sqrt{I_n-\mathcal{J}(t)}   \\ + 2[(N-n) - \operatorname{tr}(\mathcal{J}(t))](I_n - \mathcal{J}(t))dt, 
\end{multline}
where $\tilde{\mathcal{Z}}$ is the complex $n \times n$ Brownian matrix whose entries have common variance $2t$. Besides, if $(\nu_i, 1 \leq i \leq n)$ are the eigenvalues of $\mathcal{J}$, then $\nu_i < 1$ almost surely for any $ 1 \leq i \leq n$ and:
\begin{equation*}
\rho_i = \frac{1+\nu_i}{1-\nu_i}, \quad 1 \leq i \leq n,    
\end{equation*}
satisfies \eqref{NCJ1}. 
In particular,  $(\rho_i, 1 \leq i \leq n)$ is the eigenvalue process of $(I_n+\mathcal{J})(I_n-\mathcal{J})^{-1}$ which is well defined since 
\begin{align*}
I_n-\mathcal{J} = I_n - [Z^{-1}]^{\star}X^{\star}XZ^{-1} = I_n - [Z^{-1}]^{\star}[Z^{\star}Z - I_n]Z^{-1}= [Z^{-1}]^{\star}Z^{-1}    
\end{align*}
is invertible. In this regard, the relation   
\begin{equation*}
\rho_i(t) = 2\lambda_i(Nt)-1, \quad 1 \leq i \leq n,
\end{equation*}
and the fact that $(I_n+\mathcal{J})(I_n-\mathcal{J})^{-1} = 2(I_n-\mathcal{J})^{-1} - I_n$ suggest that $(\lambda_i(Nt), 1 \leq i \leq n)_{t \geq 0}$ is the eigenvalue process of $(I_n-\mathcal{J})^{-1}$. The following proposition, together with Theorem \ref{Bru-IM}, shows that this is indeed true. 
\begin{proposition}
Define
\begin{equation*}
\mathcal{I}:= (I_n-\mathcal{J})^{-1} = ZZ^{\star}.    
\end{equation*}
Then 
\begin{align*}
d \mathcal{I}(t) = \sqrt{\mathcal{I}(t)} d\tilde{\mathcal{Z}}(t) \sqrt{\mathcal{I}(t)-I_n} + \sqrt{\mathcal{I}(t)-I_n}d\tilde{\mathcal{Z}}(t)^{\star}\sqrt{\mathcal{I}(t)} + 2[N\mathcal{I}(t) -nI_n]dt. 
\end{align*}
\end{proposition}
\begin{proof}
Since $\mathcal{I}(t)\mathcal{I}(t)^{-1} = I_n$ then It\^o's formula yields
\begin{align}\label{SDE-I}
d \mathcal{I}_t & = \mathcal{I}(t) d\mathcal{J}(t)\mathcal{I}(t) + \mathcal{I}(t)(d\mathcal{J}(t))\mathcal{I}(t)(d\mathcal{J}(t))\mathcal{I}(t) \nonumber
\\& = \sqrt{\mathcal{I}(t)} d\tilde{\mathcal{Z}}(t) \sqrt{\mathcal{I}(t)-I_n} + \sqrt{\mathcal{I}(t)-I_n}d\tilde{\mathcal{Z}}(t)^{\star}\sqrt{\mathcal{I}(t)} \nonumber
\\& + 2[(N-n) - \textrm{tr}(\mathcal{J}(t))]\mathcal{I}(t)dt + \mathcal{I}(t)(d\mathcal{J}(t))\mathcal{I}(t)(d\mathcal{J}(t))\mathcal{I}(t),
\end{align}
where the second line follows from the SDE \eqref{SDE-CHG} together with the identity $\mathcal{J}(t)\mathcal{I}(t) = \mathcal{I}(t) - I_n$. 
Using the quadratic variation identities for entries of the complex Brownian motion $\tilde{\mathcal{Z}}$:
\begin{equation*}
(d\tilde{\mathcal{Z}})_{l_1l_2}(d\tilde{\mathcal{Z}}^{\star})_{l_3l_4} = 2\delta_{l_1l_4}\delta_{l_2l_3}\,dt, \quad 
(d\tilde{\mathcal{Z}})_{l_1l_2}(d\tilde{\mathcal{Z}})_{l_3l_4} = 0\,dt,
\end{equation*}
we readily deduce 
\begin{equation*}
(d\mathcal{J})_{rs}(d\mathcal{J})_{mq} = 2\left\{[I_n-\mathcal{J}]_{rq}[\mathcal{J}(I_n-\mathcal{J})]_{ms} + [I_n-\mathcal{J}]_{ms}
[\mathcal{J}(I_n-\mathcal{J})]_{rq}\right\}.
\end{equation*}
Consequently,
\begin{align*}
\mathcal{I}(d\mathcal{J})\mathcal{I}(d\mathcal{J})\mathcal{I} & = 2\left\{\textrm{tr}(\mathcal{J})\mathcal{I} + n (\mathcal{J}\mathcal{I})\right\}
\\& = 2\left\{\textrm{tr}(\mathcal{J})\mathcal{I} -nI_n +n\mathcal{I}\right\}.
\end{align*}
Plugging this finding in \eqref{SDE-I}, we get the sought SDE for the matrix-valued process $\mathcal{I}$.  
\end{proof}
Applying Theorem \ref{Bru-IM} shows that $J(Nt)$ and $\mathcal{I}(t)=Z(t)Z(t)^{\star}$ have the same eigenvalue dynamics. As a matter of fact, it is natural to wonder whether both matrix-valued processes have the same distribution. 
In the following paragraph, we compare the metrics of both underlying symmetric spaces: the cone of positive-definite Hermitian matrices and the Grassmann manifold realized as a generalized Poincar\'e disc. 


\subsection{Comparing the metrics} 
The cone of positive-definite Hermitian matrices $\mathcal{H}(N,\mathbb{C})$ is a symmetric space of non-compact type and may be identified with $GL(N,\mathbb{C})/\mathcal{U}(N)$. 
Indeed, the general linear group $GL(N,\mathbb{C})$ acts on $\mathcal{H}_N$ by
\begin{equation*}
H \in \mathcal{H}_N \mapsto gHg^{\star}, \quad g \in GL(N,\mathbb{C}).    
\end{equation*} 
The Frobenius inner product at the identity induces the affine-invariant metric
\begin{equation}\label{Metric2}
ds^2_{H} = \operatorname{tr}(H^{-1}dH\,H^{-1}dH)    
\end{equation}
on the tangent space at $H\in\mathcal H(N,\mathbb C)$.

As to the complex hyperbolic Grassmann manifold $\textrm{HG}^{n,N-n}$, it is a symmetric space as well and may be realized as a generalized Poincar\'e disc (complex matrix ball)
\begin{equation*}
 \mathbb{D}_{N-n,n} = \left\{F \in \mathbb{C}^{(N-n)\times n}: I_n - F^{\star}F > 0\right\}.   
\end{equation*}
Indeed, the indefinite unitary group $U(N-n,n)$ acts transitively on $\mathbb{D}_{N-n,n}$ by Mobius transformations as: 
\begin{equation*}
g \cdot F :=(YF+X)(WF + Z)^{-1},
\end{equation*}
where we partitioned a group element $g \in U(N-n,n)$ as:
\begin{equation*}
g = \begin{pmatrix}  Y & X \\ W & Z \end{pmatrix},    
\end{equation*}
where $Y \in \mathbb{C}^{(N-n)\times (N-n)}, X \in \mathbb{C}^{(N-n)\times n}, W \in \mathbb{C}^{n \times (N-n)}, Z \in \mathbb{C}^{n\times n}$.
In particular, the image of the null matrix $0_{N-n \times n}$ is $XZ^{-1} = w$ and the stabilizer of this action is the compact group $\mathcal{K}:= U(N-n) \times U(n)$. Therefore, the map 
\begin{equation*}
\pi: U(N-n,n) \rightarrow \mathbb{D}_{N-n,n}, \quad g \mapsto g \cdot 0_{N-n \times n}, 
\end{equation*}
yields the isomorphism 
\begin{equation*}
\textrm{HG}^{n,N-n} = U(N-n,n)/\mathcal{K} \sim \mathbb{D}_{N-n,n}.    
\end{equation*}
The Lie subalgebra 
\begin{equation*}
\mathfrak{p} := \left\{\begin{pmatrix}  0_{(N-n)\times (N-n)} & B \\ B^{\star} & 0_{n\times n}  
\end{pmatrix}, \quad B \in \mathbb{C}^{(N-n)\times n}\right\},
\end{equation*}
of $\mathfrak{u}(N-n,n)$ is the tangent space at $\mathcal{K}$ of the symmetric space $U(N-n,n)/\mathcal{K}$. Up to a positive scalar normalization, the invariant inner product restricts to the Frobenius metric $\operatorname{tr}(BB^{\star})$. The corresponding pullback metric under $\pi$ is\footnote{Composing this Hermitian metric with the complex structure gives the associated K\"ahler form, which is related to the generalized stochastic area studied in \cite{BDW}.}:  
\begin{equation}\label{Metric1}
ds_F^2 = \textrm{tr}[(I_{N-n} - FF^{\star})^{-1}dF(I_n-F^{\star}F)^{-1}dF^{\star}], \quad F = g\cdot 0_{N-n \times n} = XZ^{-1}.   
\end{equation}
Both metrics \eqref{Metric1} and \eqref{Metric2} reduce to the Frobenius metric at their base points. Away from those points, however, $ds_F^2$ contains both $(I_{N-n}-FF^{\star})^{-1}$ and $(I_n-F^{\star}F)^{-1}$. Their nonunit eigenvalues agree, but their eigenspaces are the left and right singular-vector spaces of $F$. By contrast, \eqref{Metric2} is governed by the eigenspaces of one positive-definite matrix. The corresponding full generators are therefore different, even though their radial parts reduce to the same particle generator after the time normalization above.

\subsection{An algebraic correspondence}
The previous comparison of metrics has a probabilistic flavor. Instead, there is an algebraic correspondence between the matrices defining the stochastic process $\mathcal{I}$ and $J$, which relies on the Schur complement displayed in \eqref{Def-R} as follows. More precisely, let $A$ be a fixed 
positive-definite matrix as in the proof of Lemma \ref{Positivity} and specialize $F=R^{\star}\in\mathbb{C}^{(N-n)\times n}$. Then \eqref{R-contraction} shows $F^{\star}F < I_n$, in other words $F$ belongs to the matrix ball. 
We may then define the following matrix
\begin{equation}\label{Canonical-lift}
\begin{pmatrix}
(I_{N-n}-FF^{\star})^{-1/2} & F(I_n-F^{\star}F)^{-1/2}\\
F^{\star}(I_{N-n}-FF^{\star})^{-1/2} & (I_n-F^{\star}F)^{-1/2}
\end{pmatrix},
\end{equation}
and prove that it is an element of $U(N-n,n)$. Denoting $Z_F$ its lower-right block, then
\begin{equation}\label{Pointwise-link}
Z_FZ_F^{\star} = (I_n-RR^{\star})^{-1}.
\end{equation}
But equations \eqref{Schur-J} and \eqref{Pointwise-link} show that $\pi_n(A)\pi_n(A^{-1})$ is similar to $Z_FZ_F^{\star}$. Thus there is an exact pointwise spectral correspondence. For $A=H(t)$, however, the induced process $F(t)=R(t)^{\star}$ need not be Brownian motion for the Grassmannian metric.

\section{The large-size limit of $J$} 
Free probability theory provides an operator-algebraic framework for describing large-dimensional limits of random matrices through their moments (that is, normalized traces of products of powers) and, more generally, through operator norms of suitable functions of these matrices. The theory is formulated in terms of a unital von Neumann algebra $\mathcal{A}$ endowed with a faithful normal tracial state $\tau$, normalized by $\tau({\bf 1})=1$. The pair $(\mathcal{A},\tau)$ is a noncommutative probability space, and freeness plays the role of independence. We refer to \cite{Nic-Spe} for a comprehensive introduction.

Many self-adjoint matrix-valued stochastic processes converge, either in moments or in operator norm, to their free analogues. For example, Dyson's Hermitian Brownian motion converges to the free additive (or semicircular) Brownian motion. Likewise, the complex Brownian matrix and the complex Wishart (or Laguerre) process converge to the free circular Brownian motion and the free Wishart process, respectively \cite{Cap-Don}.

Building on the free stochastic calculus developed by Biane and Speicher \cite{Bia-Spe}, one defines the free unitary Brownian motion $(u_t)_{t\ge0}$ as the unique solution of a free stochastic differential equation. As shown in \cite{Biane}, this process arises as the large-dimensional limit of the unitary Brownian motion. Furthermore, the radial part of the compression of $u_t$ by two free orthogonal projections gives rise to the free Jacobi process, which itself is the large-dimensional limit of the Hermitian Jacobi process \cite{Demni}.

In the present section, we introduce the free non-compact Jacobi process as the large-N limit of the matrix-valued process $J$, assuming $n = n(N)$ such that
\begin{equation*}
\lim_{N \rightarrow +\infty} \frac{n(N)}{N} = \kappa \in (0,1].    
\end{equation*} 
At the operator-algebraic level, this process is the product of the compression of the squared radial part $(h_t)_{t\ge0}$ of free multiplicative Brownian motion and the compression of $h_t^{-1}$. Each fixed-time operator is bounded, but the terminology \emph{non-compact} records the geometry and the unbounded radial state space of the finite-dimensional model. Before defining the process, we recall some facts about free multiplicative and free positive Brownian motions; see \cite{Cebron} and the references therein.

\subsection{The free multiplicative and the free positive Brownian motions} 
Let $(\mathcal{A}, \tau)$ be a non-commutative probability space. The (left) free multiplicative Brownian motion $(g_t)_{t \geq 0}$ is the unique solution of 
the free SDE 
\begin{equation*}
dg_t = g_tdc_t, \quad g_0 = 1_{\mathcal{A}}, 
\end{equation*}
where $(c_t)_{t \geq 0}$ is a free circular Brownian motion. The free positive Brownian motion is defined by $h_t: = g_tg_t^{\star}, t \geq 0,$ and 
satisfies the free SDE: 
\begin{equation*}
dh_t = \sqrt{2} g_t\frac{dc_t + dc_t^{\star}}{\sqrt{2}} g_t^{\star} + h_t dt. 
\end{equation*}
Obviously, $[(c_t+c_t^{\star})/\sqrt{2}]_{t \geq 0}$ is a semi-circular Brownian motion. Moreover, using the polar decomposition $g_t = \sqrt{h_t}v_t$, where $v_t$ is a unitary operator together with the free analogue of L\'evy's characterization (see Lemma 4.1 in \cite{Demni}), we get the autonomous free SDE:
\begin{equation*}
dh_t = \sqrt{2}\sqrt{h_t}dn_t \sqrt{h_t} + h_t dt,    
\end{equation*}
where $(n_t)_{t \geq 0}$ is a semi-circular Brownian motion. The moments $r_k(t):= \tau(h_t^k), k \geq 0,$ of $h_t$ are easily seen to satisfy: 
\begin{equation*}
\frac{dr_k}{dt}(t) = -kr_k(t) + k \sum_{s=0}^kr_s(t)r_{k-s}(t) = kr_k(t) + k\sum_{s=1}^{k-1}r_s(t)r_{k-s}(t), \quad r_0(t) = 1.     
\end{equation*}
By virtue of Lemma 13 and of Proposition 11 in \cite{Biane1}, this moment sequence uniquely determines a compactly-supported measure $\nu_t$ which is absolutely continuous with respect to Lebesgue measure on $\mathbb{R}_+$ and whose support is the interval 
\begin{equation*}
[(\sqrt{t+1}-\sqrt{t})^2e^{-\sqrt{t(1+t)}}, (\sqrt{t+1}+\sqrt{t})^2e^{\sqrt{t(1+t)}}].  
\end{equation*}
As a matter of fact, $h_t$ is invertible for any time $t \geq 0$, and its inverse solves the free SDE 
\begin{equation*}
dh_t^{-1} = - \sqrt{2}\sqrt{h_t^{-1}}dn_t \sqrt{h_t^{-1}} + h_t^{-1} dt. 
\end{equation*}
It is also known that (see \cite{Biane}, Proposition 5):
\begin{equation*}
r_k(t) = e^{kt} Q_k(-2t)    
\end{equation*}
where $kQ_k(u) = L_{k-1}^{(1)}(u),\, k \geq 1,$ is a Laguerre polynomial.

\subsection{The free compression of $h_t$ by an orthogonal projection}
Let $P$ be an orthogonal projection of trace $\tau(P) = \kappa$ and assume $P$ and $(h_t)_{t \geq 0}$ are free in $(\mathcal{A}, \tau)$. The self-adjoint operator $Ph_tP$ is the free compression of $h_t$ by $P$. By \cite[Theorem 4.6]{Cebron}, it is the large-size limit, in moments, of
\begin{equation*}
P_{n(N)}H(t)P_{n(N)}, \quad t \geq 0,  
\end{equation*}
in the non-commutative probability space $(\mathbb{M}_N(\mathbb{C}), (1/N)\textrm{tr})$. We may consider its spectral distribution in the compressed probability space 
\begin{equation*}
\left(P\mathcal{A}P, \frac{\tau}{\tau(P)}\right),    
\end{equation*}
whose unit is $P$. This probability distribution is compactly supported and is therefore characterized by its moments. 
\begin{proposition}
For any $k \geq 1$ and $t \geq 0$, 
\begin{equation*}
\frac{\tau[(Ph_tP)^k]}{\kappa} = e^{kt} Q_k(-2\kappa t). 
\end{equation*}    
\end{proposition}
\begin{proof}
The proof mimics the first proof of Proposition 2.1 in \cite{Demni-Hamdi}. Indeed, the process $(\tilde{h}_t:= e^{-t}Ph_tP)_{t \geq 0}$ satisfies 
\begin{equation*}
d\tilde{h}_t = \sqrt{2}P\sqrt{e^{-t} h_t}dn_t \sqrt{e^{-t} h_t}P. 
\end{equation*}
Applying the free It\^o's formula \cite[Proposition 4.3.2]{Bia-Spe} to the monomial $z^k$ and taking the trace, we deduce that the sequence 
\begin{equation*}
f_k(t):= e^{-kt}\tau[(Ph_tP)^k], \quad k \geq 1,    
\end{equation*}
satisfies the same differential system as $e^{-kt}r_k(t) = Q_k(-2t), k \geq 1,$: 
\begin{equation*}
\frac{df_k}{dt}(t)= k\sum_{s=1}^{k-1}f_s(t)f_{k-s}(t), 
\end{equation*}
with initial conditions $f_k(0)=\kappa$ for every $k\geq1$. Setting $a_k=f_k/\kappa$, the proposition follows from uniqueness for the differential system
\begin{equation*}
\frac{da_k}{dt}(t)= \kappa k\sum_{s=1}^{k-1}a_s(t)a_{k-s}(t).  
\end{equation*}
subject to $a_k(0)=1$ for every $k\geq1$.  
\end{proof}
From this proposition, it follows that
\begin{equation*}
\frac{\tau[(Ph_tP)^k]}{\kappa} = e^{k(1-\kappa)t} \tau[(h_{\kappa t})^k] = \tau[(e^{(1-\kappa)t}h_{\kappa t})^k].   
\end{equation*}
\begin{corollary} 
The spectral distribution of $Ph_tP$ in the compressed algebra $(P\mathcal{A}P, \tau/\kappa)$ coincides with the spectral distribution of 
$e^{(1-\kappa)t}h_{\kappa t}$ in $(\mathcal{A}, \tau)$. 
\end{corollary}
\begin{remark}
The pushforward of $\nu_t$ by the logarithm map $u \mapsto \log(u)$ may be represented as: 
\begin{equation*}
 \mu_{\textrm{SC}, 2\sqrt{2t}} \boxplus \textrm{Unif}_{[-t,t]}   
\end{equation*}
where $\boxplus$ is the additive free convolution, $\mu_{\textrm{SC}, 2\sqrt{2t}}$ is the semi-circle distribution of variance $2t$ and
$\textrm{Unif}_{[-t,t]}$ is the uniform distribution on $[-t,t]$. Consequently, the pushforward of the spectral distribution of $Ph_tP$ in 
$(P\mathcal{A}P, \tau/\kappa)$ by $u \mapsto \log(u)$ admits the representation:      
\begin{equation*}
\delta_{(1-\kappa) t} \boxplus \mu_{\textrm{SC}, 2\sqrt{2\kappa t}} \boxplus \textrm{Unif}_{[-\kappa t,\kappa t]}.   
\end{equation*}
\end{remark}

\subsection{The free non-compact Jacobi process}
Assume $n=n(N)$ and $n/N \rightarrow \kappa$ as $N \rightarrow +\infty$; equivalently,
\begin{equation*}
\frac{1}{N}\operatorname{Rank}(P_{n(N)})\longrightarrow\tau(P)=\kappa.
\end{equation*}
Theorem 4.6 of \cite{Cebron} then yields that the $N \times N$ matrix-valued process 
\begin{equation*}
P_{n(N)}H(t)P_{n(N)}H^{-1}(t)P_{n(N)}, \quad t \geq 0    
\end{equation*}
converges in the sense of moments to $Ph_tPh_t^{-1}P$. But since 
\begin{equation*}
\frac{1}{n(N)}\mathbb{E}\left[\operatorname{tr}(J(t)^k)\right] = \frac{1}{N} 
\frac{N}{n(N)} \mathbb{E}\left[\textrm{tr}(P_{n(N)}H(t)P_{n(N)}H^{-1}(t)P_{n(N)})^k\right],   
\end{equation*}
it follows that, for every $t \geq 0$, $J(t)$ converges in moments in the compressed probability space to $Ph_tPh_t^{-1}P$.
\begin{definition}
The process $j_t:= Ph_tPh_t^{-1}P$, valued in the compressed probability space, is called the free non-compact Jacobi process.    
\end{definition}
The second proof of Lemma \ref{Positivity} may be adapted to the infinite dimensional setting and yields: 
\begin{lemma}
For any $t \geq 0$, the spectrum $\sigma(j_t) \subset [1,+\infty)$ in the compressed space. 
\end{lemma}
\begin{proof} 
The statement is trivial for $t=0$ since $j_0 = P$ is the unit of the compressed space $(P\mathcal{A}P, \tau/\kappa)$. 
Let $t > 0$ and let $\mathcal{H}$ be the Hilbert space on which $h_t$ acts on. Then we may decompose 
\begin{equation*}
 \mathcal{H} = \operatorname{Ran}(P) \oplus \ker(P)   
\end{equation*}
and write
\begin{equation*}
h_t = \left(\begin{matrix}
  x_t & y_t \\ 
  y_t^{\star} & z_t
\end{matrix}\right)
\end{equation*}
so that $Ph_tP = x_t$ is positive on $P\mathcal{H}$. Since $h_t$ is bounded below by a positive constant, $z_t = P^{\perp}h_tP^{\perp}$ is positive and invertible on $P^{\perp}\mathcal{H}$. We may therefore define the Schur complement (see \cite{CMS} and references therein)
\begin{equation*}
s_t:= x_t - y_tz_t^{-1}y_t^{\star}
\end{equation*}
on $P\mathcal{H}P$, which is positive since $h_t$ is so. It follows that $Ph_t^{-1}P = s_t^{-1}$, and in turn $j_t = x_ts_t^{-1}$, on $P\mathcal{H}P$. 
Furthermore, it is obvious that $x_t-s_t = y_tz_t^{-1}y_t^{\star}$ is non negative therefore $s_t^{-1} - x_t^{-1}$ is so as well. 
Equivalently, the positive representative
\begin{equation}\label{Positive-representative}
\widehat j_t:=x_t^{1/2}s_t^{-1}x_t^{1/2}
\end{equation}
satisfies $\widehat j_t-P\geq0$. Moreover,
\begin{equation*}
\sigma(x_t^{1/2}s_t^{-1}x_t^{1/2}) = \sigma(x_ts_t^{-1}) = \sigma(j_t),  
\end{equation*}
as operators acting on $P\mathcal{H}$. As a result, $\sigma(j_t) \subset [1,+\infty)$. Although $j_t$ need not be self-adjoint, it is similar to the positive operator $\widehat j_t$, and $\tau(j_t^k)=\tau(\widehat j_t^k)$. Hence the spectral distribution of $j_t$ below means the spectral distribution of this positive representative.
\end{proof}

The free analogue of the SDE \eqref{Mat-SDE} is straightforward and reads (\cite{Auer-PhD}, proof of Theorem 5.6.2): 

\begin{equation}\label{FreeJac-SDE}
dj_t  = [\sqrt{2}Ph_t^{1/2}]dn_t[h_t^{1/2}Ph_t^{-1}P] - [Ph_tPh_t^{-1/2}]dn_t[\sqrt{2}h_t^{-1/2}P]  + 2\left[j_t - \kappa P\right]dt.  
\end{equation}
By using \cite[Proposition 4.3.2]{Bia-Spe}
we get the following ordinary differential equation (ODE) for the moments of $j_t, t \geq 0$: 
\begin{theorem}
For any $k \geq 1$, let 
\begin{equation*}
m_k(t):= \frac{\tau(j_t^k)}{\tau(P)}    
\end{equation*}
and set $m_0(t) = 1$. Then, the map $t \mapsto m_k(t)$ is differentiable in $(0,+\infty)$ for any $k \geq 1$ and satisfies: 
\begin{equation}\label{Moments-ODE}
\frac{dm_k}{dt}(t) = 2k\left[m_k(t) - \kappa m_{k-1}(t) + 
\kappa\sum_{s=1}^{k-1}m_s(t)(m_{k-s}(t)-m_{k-s-1}(t))\right],    
\end{equation}
where any empty sum is zero. 
\end{theorem}
\begin{proof}
Write $dj_t=dM_t+2(j_t-\kappa P)dt$, where $dM_t$ is the sum of the two stochastic integrals in \eqref{FreeJac-SDE}. The free It\^o product rule gives
\begin{equation*}
d(j_t^k)=\sum_{p=0}^{k-1}j_t^p(dj_t)j_t^{k-1-p}
+\sum_{0\leq p<q\leq k-1}j_t^p(dM_t)j_t^{q-p-1}(dM_t)j_t^{k-1-q}.
\end{equation*}
Use the semicircular contraction rule
\begin{equation*}
(a\,dn_t\,b)(c\,dn_t\,d)=a\,\tau(bc)\,d\,dt
\end{equation*}
for the four products obtained by expanding the two summands of $dM_t$. Cyclicity of $\tau$, the identities $Pj_t=j_tP=j_t$, and freeness of $P$ and $h_t$ reduce the sum of all quadratic terms to
\begin{equation*}
\sum_{0\leq p<q\leq k-1}\tau\!\left(j_t^p(dM_t)j_t^{q-p-1}(dM_t)j_t^{k-1-q}\right) =2k\kappa^2\sum_{s=1}^{k-1}m_s(t)\bigl(m_{k-s}(t)-m_{k-s-1}(t)\bigr)dt.
\end{equation*}
The drift terms contribute $2k\kappa[m_k(t)-\kappa m_{k-1}(t)]dt$. Dividing their sum by $\kappa=\tau(P)$ proves \eqref{Moments-ODE}.
\end{proof}
This ODE looks quite similar to the one derived in \cite{Demni}, Corollary 2, and satisfied by the moments of the free Jacobi process. 
The sign difference in \eqref{Moments-ODE} reflects that the spectral support extends beyond one. For each fixed $t$, the operator remains bounded, while the growth of $m_k(t)$ as $k\to\infty$ records the upper edge of its support. Induction also shows that, for every fixed $k\geq1$, $m_k(t)$ is a quasi-polynomial in $(t,e^t)$.

In \cite{DHH}, Corollary 3.3, we proved that if $\tau(P) =1/2$ then the spectral distribution of the corresponding free Jacobi process at any time $t$ is the pushforward of the spectral distribution of the free unitary Brownian motion at time $2t$ under the Sz\"ego map: 
\begin{equation*}
z \mapsto \frac{z+z^{-1}+2}{4}, \quad |z| = 1.     
\end{equation*}
An analogous result holds true for the free non-compact Jacobi process: 
\begin{proposition}
If $\kappa = \tau(P) = 1/2$, then the spectral distribution of the corresponding free non-compact Jacobi process at any time $t$ is the pushforward of $\nu_{2t}$ under map: 
\begin{equation*}
x \mapsto \frac{x+x^{-1}+2}{4}, \quad x > 0.     
\end{equation*}
\end{proposition}
\begin{proof}
Write $P := ({\bf 1}+S)/2$ so that 
\begin{equation*}
Ph_tPh_t^{-1} = \frac{({\bf 1}+S)h_t({\bf 1}+S)h_t^{-1}}{4} = \frac{({\bf 1}+S)({\bf 1}+S_t)}{4},   
\end{equation*}
where $S_t = h_tSh_t^{-1}$. Now, since $S^2= S_t^2 = {\bf 1}$ then we can prove that 
\begin{equation*}
2\tau(j_t^k) = 2 \tau[(Ph_tPh_t^{-1})^k]= \frac{1}{4^k}\left\{\binom{2k}{k} + 2\sum_{s=1}^k\binom{2k}{k-s}\tau[(SS_t)^s]\right\}.    
\end{equation*}
But $SS_t$ is an invertible operator; therefore 
\begin{align*}
2\sum_{s=1}^k\binom{2k}{k-s}\tau[(SS_t)^s] & = \sum_{s=1}^k\binom{2k}{k-s}\left[\tau[(SS_t)^s] + \tau[(SS_t)^{-s}]\right],
\end{align*}
whence 
\begin{equation*}
2\tau(j_t^k) = \frac{1}{4^k}\sum_{s=-k}^k\binom{2k}{k+s}\tau[(SS_t)^s] = \frac{1}{4^k}\sum_{s=0}^{2k}\binom{2k}{s}\tau[(SS_t)^{s-k}].    
\end{equation*}
Using the linearity of $\tau$ together with the binomial Theorem, we further get 
\begin{equation}\label{Equation1}
2\tau(j_t^k) = \frac{1}{4^k}\tau\left[(SS_t + (SS_t)^{-1} +2{\bf 1})^k\right].    
\end{equation}
Finally, since $\tau(S) = 0$ and since $S$ is $\star$-free from $(h_t)_{t \geq 0}$ then Lemma 3.8 in \cite{HL} shows that $h_t^{-1}$ and $Sh_tS$ are free as well. The operator $Sh_tS$ is positive, and its spectral distribution is $\nu_t$. Consequently, 
\begin{equation*}
 \tau[(SS_t)^s] = \tau[(\tilde{h}_th_t^{-1})^s],   
\end{equation*}
where $\tilde{h}_t$ is a free copy of $h_t$. The law $\nu_t$ is invariant under $x\mapsto x^{-1}$, so $h_t^{-1}$ also has law $\nu_t$. Since $(\nu_t)_{t\geq0}$ is a semigroup for free multiplicative convolution, $\tilde h_t h_t^{-1}$ has law $\nu_t\boxtimes\nu_t=\nu_{2t}$. Hence $\tau[(SS_t)^s]=\tau[(h_{2t})^s]$, and \eqref{Equation1} proves the proposition.
\end{proof}
\begin{remark}
The unitary analogue of the process $SS_t$ is $Su_tSu_t^{\star}$ where $(u_t)_{t \geq 0}$ is the free unitary Brownian motion. It arose in the study of the free Jacobi process, and the Herglotz transform of its spectral distribution turns out to be closely related to the radial L\"owner equation \cite{Hamdi}. Analogously, one may show that the Stieltjes transform of the spectral distribution is closely related to the radial L\"owner equation. 
\end{remark}
\bibliographystyle{alpha}
\bibliography{reference.bib}
\end{document}